\documentclass[11pt]{article}
\usepackage{fullpage,graphicx,amsmath,amsthm,amsfonts,subfigure,tikz,url}
\usepackage{times}
\usepackage{enumitem}
\usepackage{amsmath,amsthm,amsbsy,amsfonts,amssymb,dsfont,stmaryrd}
\usepackage{color}

\usepackage[normalem]{ulem}

\newcommand{\ignore}[1]{{}}

\newcommand{\suppress}[1]{}
\theoremstyle{plain}
\newtheorem{theorem}{Theorem}%[section]
\newtheorem{lemma}[theorem]{Lemma}

\newtheorem{proposition}[theorem]{Proposition}
\theoremstyle{definition}

\theoremstyle{remark}
\newtheorem{remark}[theorem]{Remark}
\newcommand{\zset}{\mathbb Z}

\newcommand{\rset}{\mathbb R}

\numberwithin{equation}{section}

\newcommand{\cA}{{\mathcal A}}

\newcommand{\cH}{{\mathcal H}}
\newcommand{\cL}{{\mathcal L}}

\newcommand{\cS}{{\mathcal S}}
\newcommand{\cT}{{\mathcal T}}

\newcommand{\be}{\begin{equation}}
\newcommand{\ee}{\end{equation}}
\newcommand{\bea}{\begin{eqnarray}}
\newcommand{\eea}{\end{eqnarray}}
\newcommand{\bean}{\begin{eqnarray*}}
\newcommand{\eean}{\end{eqnarray*}}
\newcommand{\da}{\dagger}

\DeclareMathOperator{\spec}{spec}

\usepackage[bb=boondox]{mathalfa}
\newcommand{\one}{\mathbb{1}}

\renewcommand{\cS}{{\mathcal Q}}
\newcommand{\cSin}{{\mathcal S}}

\newcommand{\hrf}{h^{\text{ref}}}
\newcommand{\Hrf}{\cH^{\text{ref}}}
\newcommand{\murf}{\mu^{\text{ref}}}

\newcommand{\tmix}{\tau_{\rm mix}}
\newcommand{\tcoup}{\tau_{\rm coup}}
\newcommand{\hit}{\tau_{\rm hit}}
\newcommand{\TVD}[2]{\Vert #1 - #2 \Vert_{\rm TV}}

\newenvironment{proofof}[1]{{\par\bigskip\noindent \em Proof of #1.}\/}{\hfill\qed\bigskip}

\usepackage{hyperref,xcolor}
\hypersetup{
  colorlinks=true,
  linkcolor=blue,
  linktoc=all,
  citecolor=blue,
  urlcolor=blue,
  bookmarksnumbered=true
}
\usepackage{mathtools}
\mathtoolsset{showonlyrefs}
\begin{document}
\title{Mixing times and spectra of non-equilibrium symmetric exclusion processes on general graphs}

\author{Leonard J. Schulman\thanks{Division of Engineering and Applied Science, California Institute of Technology, Pasadena CA 91125, USA. schulman@caltech.edu. Supported in part by NSF grant CCF-2321079.} \and
             Alistair Sinclair\thanks{Computer Science Division, University of California, Berkeley CA 94720, USA.  sinclair@cs.berkeley.edu.  Supported in part by NSF grant CCF-223109.}}

\date{\today}

\maketitle
\thispagestyle{empty}

\abstract{The symmetric exclusion process (SEP) is a classical model of interacting particles on a graph,
in which multiple particles execute random walks subject to the constraint that no two particles may occupy
the same vertex.  The conservative version of the process, in which the number of particles is fixed, is by
now very well understood, as are numerous other equilibrium models in statistical physics, through the study
of reversible Markov chains.   

In this paper,
we study the \emph{non-equilibrium} version of the SEP, in which some vertices of the graph interact
with external ``heat baths'', held at different temperatures, that mediate transport through the graph.  
The resulting steady-state distribution
is still the stationary distribution of a natural Markov chain, but the chain is no longer reversible.  This means
that even the steady-state distribution is very hard to describe, and indeed is the subject of a rich
literature in statistical physics, which has focused almost exclusively on finite subsets of the lattice~$\zset^d$.  
Very little is known about other important quantities, notably the {\it mixing time}, i.e., the time to reach steady-state.

We derive various quantitative results about the non-equilibrium SEP for arbitrary graphs with
arbitrary heat baths.  Our main result is a bound
on the mixing time in terms of the worst-case hitting time of a single particle to any of the heat baths, which is a
much more accessible quantity; this bound is tight up to two logarithmic factors in the size of the graph.  We
then use this result to derive a simple algorithm, running in time roughly $n^{O(k)}$, that computes the steady-state
joint occupation distribution for any set of $k$ vertices; given that the steady-state distribution is very elusive,
this provides a potentially useful tool to study its $k$-wise correlations for small values of~$k$.  Finally, we prove
that the spectrum of the Markov chain associated with the non-reversible dynamics for the SEP is independent
of the heat bath temperatures; thus in particular the spectrum is always real, despite the fact that the rate matrix
of the dynamics is not reversible, or even diagonalizable.}
\newpage
\setcounter{page}{1}

\section{Introduction}
A major focus in statistical mechanics is the computation of properties of the Gibbs or \emph{equilibrium} steady-state distribution of a physical system (generally expressible through a local Hamiltonian). For this reason, over the last four decades, extensive effort (much of it in theoretical computer science) has been devoted to the development of algorithms, particularly Markov Chain Monte Carlo (MCMC), to sample from such Gibbs distributions.  The resulting Markov chains (e.g., Glauber dynamics) are naturally reversible with respect to the Gibbs distribution.

Equilibrium steady states are, however, but a narrow portion of the physical states of interest. The steady state of a typical physical system is not an equilibrium: rather, in many cases the system is held out of equilibrium by external boundary conditions. For example, a body of water often has a salinity gradient; an insulating barrier typically has a temperature gradient, and so on. In such systems, in steady state there is continual transport (of salt, heat etc.).  The steady-state distribution is again the stationary distribution of a natural Markov chain, which however is {\it non-reversible}.  In contrast to the equilibrium setting, where there is an extensive literature analyzing the associated reversible Markov chains, the analysis of the dynamics for non-equilibrium systems has received relatively little treatment. The current paper is devoted to remedying this deficiency in one famous system from mathematical physics: the symmetric exclusion process (SEP).

The SEP~\cite{Spitzer70,Liggett85} is a classical model of interacting particles
in which multiple unlabeled particles execute random walk on a connected, undirected 
graph subject to the constraint that no two particles may occupy the same vertex.  Equivalently, since the particles are unlabeled,
we may think of an {\it exchange\/} process driven by independent Poisson clocks
on the edges of the graph: when the clock on edge $\{u,v\}$ rings, the contents of vertices $u,v$ are exchanged.
In the classical, conservative model the number, $k$, of particles is conserved and the dynamics is
reversible with equilibrium distribution that is uniform over all $n\choose k$ possible choices of occupied 
vertices (where $n$ is the number of vertices in the graph).  
The mixing time of this dynamics has received much attention over many
years and is by now very well understood (see, e.g., \cite{Morris06,CLR10,Oliveira13,Lacoin1,Lacoin2,HP20}).

In this paper we study the much less well understood {\it non-equilibrium\/} setting, in which in addition an arbitrary
subset of vertices interacts with external ``heat baths": these can be viewed as reservoirs with fixed occupation
probabilities (which we refer to as ``temperatures"~$h_v \in [0,1]$).  When a vertex~$v$ interacts with its heat bath (triggered again
by a Poisson clock), an exchange is performed between~$v$ and the bath, which is at any time occupied with
fixed probability~$h_v$.  Thus the heat baths continuously inject or remove particles from the system at a rate
that depends on their clocks and their temperatures.  In this scenario, the SEP is in general no longer reversible, 
and is held out of equilibrium by the heat baths, which mediate net transport across the system.  
The system still converges to a steady-state distribution, which however is typically very
hard to describe due to the absence of reversibility.

The non-equilibrium SEP---as well as its asymmetric variants---has been extensively investigated, especially
in the statistical physics community.  However, the focus there has been almost exclusively on the special case 
where the graph is a finite subset of~$\zset^d$, and indeed even more specifically on the one-dimensional case of
a finite path.  Moreover, the main emphasis has been on deriving structural properties of the steady-state
distribution, such as correlations between vertex occupancies, and transport rates: even in the simple case
of a path with one heat bath at each endpoint, there is no closed-form expression for the steady-state
distribution (though a classical paper of Kingman~\cite{Kingman69} derives an expression for the joint occupation
probability for any $k$ vertices).  For an account of the rich literature on these topics, we refer the reader to
\cite{Spohn83,BertiniDGJL03,BertiniDGJL07,Derrida07,ChouMZ11} and the references therein.

In this paper we focus on the {\it mixing time\/} of the non-equilibrium SEP, a question that has so far
received little attention, with the exception of two recent papers.  
Goncalves {\it et al.}~\cite{Goncalves} and 
Nestoridi {\it et al.}~\cite{Nestoridi} give very precise bounds on the mixing time (including the cutoff phenomenon)
for the exclusion process on a \emph{path} with heat baths at the endpoints.  In a closely related recent
paper, Salez~\cite{Salez23} obtains 
similarly tight results for the mixing behavior of the (equilibrium) SEP on {\it arbitrary\/} graphs with heat baths, but under the
strong assumption that all heat baths are held at the {\it same\/} temperature: this results in a reversible dynamics,
which is why the system is still in equilibrium. The steady-state distribution here is just a product distribution
over the vertices. 

All the above leaves unaddressed the general case of arbitrary heat baths on non-path graphs,
in which the dynamics is non-reversible and the steady-state distribution is much more complex. In this paper we
give two-sided bounds on the mixing time of the non-equilibrium SEP on general graphs, as well as related
results on computing the steady-state distribution and on the spectrum of the generator of the SEP.  
Our main result is a bound on the mixing time in terms of the
maximum expected hitting time~$\hit$ of a single random walk on the graph from any vertex to a heat bath; 
this quantity readily yields an upper bound on mixing time, but it is more delicate to show that this upper bound
is in fact tight up to logarithmic factors in the size of the graph. We emphasize that our bounds are independent
of the temperatures of the heat baths, and dependent only on their locations: in particular, the mixing time for
any choice of heat bath
temperatures is (up to log factors) the same as for the case when all temperatures are equal, which is the reversible
case treated by Salez~\cite{Salez23}. However, unlike Salez,
we do not establish a cut-off phenomenon for the non-reversible dynamics. 

Next we leverage this result to obtain an algorithm with running time $n^{O(k)}\log \hit$ that, for any graph
and for any heat baths, 
computes the steady-state joint distribution of any subset of $k$~vertices to any specified degree of accuracy.
This algorithm is based on an extension of Liggett's duality theory for exclusion processes~\cite{Liggett85} 
to allow for the presence of heat baths. We believe that this algorithm, for small values of~$k$, 
may be useful in investigating properties of the steady-state distribution, in networks where closed-form
estimates are unavailable (which, to the best of our knowledge, includes all cases except the path
and some subsets of~$\zset^d$ with appropriately positioned heat baths).

Finally, we prove a result of a similar flavor for the spectrum of the Markov chain: namely, 
the spectrum of an SEP with heat baths is independent of the heat bath temperatures.  Since this
includes the reversible case (equal-temperature heat baths), where the spectrum is real, we deduce
that the spectrum of an arbitrary non-equilibrium SEP is always real.  This result seems surprising
in light of the fact that these Markov chains are not only not reversible, but in general not even
diagonalizable. In the reversible case treated by Salez~\cite{Salez23}, the spectrum dictates the mixing time;
however, for non-reversible Markov chains it is known that the spectrum may give essentially no
information about the mixing time~\cite{MehtaS20}, so our spectral result is not immediately
connected to our mixing time result.  We leave potential applications as an open question; in particular,
we point out that, while our result on the spectrum makes use of the continuity of eigenvalues
as a function of the transition matrix, allowing us to relate the system with arbitrary heat baths
to the reversible case, such continuity is not shared by the eigenvectors and thus our method
apparently does not directly shed light on the steady-state distribution of the non-reversible system.

%%%%%%%%%%%%%%%%%%%%%%%%%%%%%%%%%%%%%%%%%%%%%%%%%%%%%

\section{Framework and Results} \label{sec:fmwkresults}
\subsection{Framework}\label{subs:fmwk}
Let $G=(V,E)$ be a connected undirected graph with vertex set~$V$ of size~$n$, and edge weights
$\eta_{\{u,v\}}>0$ on each edge~$\{u,v\}$.  The vertices of~$G$ are the sites of the system, and interact
with their neighbors at rates specified by the edge weights.  A subset $V'\subseteq V$ of sites also
interact with external heat baths at rates $\xi_v> 0$ for $v\in V'$.  To avoid trivialities we always assume that $V'$ is non-empty.

At any given time, sites in the SEP may be either occupied by a single particle or unoccupied.  Thus the
state of the process at time~$t$ is described by a vector $\sigma_t=\{\sigma_t(v)\}\in \{0,1\}^V$ which we
call a {\it configuration}, where the random
variable $\sigma_t(v)$ takes the value~1 or~0 if vertex~$v$ is occupied or unoccupied, respectively.  This configuration
evolves in time under the following continuous time exclusion dynamics.  Each edge~$\{u,v\}$ is equipped
with an independent Poisson clock with rate~$\eta_{\{u,v\}}$; when the clock rings, the values $\sigma(u),\sigma(v)$ are
exchanged: i.e., if there is a particle at~$u$ it moves to~$v$ and vice versa.  Thus the observable state of
the system is affected by such an exchange operation only if exactly one of $u,v$ is occupied at the time
of the operation.

The connection of each site $v\in V'$ with its external heat bath is also governed by an independent Poisson
clock of rate~$\xi_v$; when this clock rings, an exchange
operation is performed between~$v$ and the heat bath, which is at any time unocuppied or occupied
according to a fixed 2-point probability distribution $h_v=(h_v(0),h_v(1))$, with $h_v(0)=1-h_v(1)\in [0,1]$.
The state of the heat bath is unaffected by such an exchange.  We make the further non-triviality assumption
that there exist (not necessarily distinct) heat baths $u,v$ such that $h_u(0)>0$ and $h_v(1)>0$.  (Thus we
rule out the degenerate case in which all heat baths are uniformly occupied, respectively uniformly unoccupied:
in these cases the dynamics absorbs into the fully occupied (respectively, unoccupied) configuration.)

We note that, in the absence of heat baths, the above dynamics coincides with the classical, closed SEP,
where the number of particles in the system is conserved over time; the dynamics is then also reversible
w.r.t.\ the uniform steady-state distribution over all legal placements of the particles.  In the presence of heat baths,
the dynamics becomes in general non-reversible as transport occurs between heat baths,
and its steady-state distribution~$\mu$ may consequently be very hard to describe.  
However, in the special case that all heat baths are identical and non-trivial (i.e., the 2-point probability distribution
at every vertex~$v$ with $\xi_v>0$ is $h_v=(h(0),h(1))$ for some fixed~$h$ with $h(0)=1-h(1)\in (0,1)$),
it is easily checked (see Proposition~\ref{prop:basic})
that the dynamics is ergodic and reversible w.r.t.\ the steady-state product
distribution $\mu(\sigma)=\prod_{v\in V} h(\sigma(v))$.  

For later reference in our study of the spectrum of this process, we now give an equivalent algebraic description.
Let $S$ denote the $\{0,1\}^2\times\{0,1\}^2$ transition matrix corresponding to the exchange operation on 
a pair of adjacent vertices, i.e., if the values of the configuration at the vertices before and after the exchange are $a,b$ and $c,d$,
respectively, then
\be S_{a,b}^{c,d} := \begin{cases} 1 & \text{ if } a=d \text{ and } b=c; \\ 
0 & \text{ otherwise. }\end{cases} \label{exchop} \ee
(Here, the subscript is a row index and the superscript is a column index.)  We write $S(\{u,v\})$ for $S$
applied to the edge $\{u,v\}$, i.e., $S(\{u,v\})$ exchanges the values of $\sigma(u)$ and~$\sigma(v)$.

We represent the heat-bath interactions using another set of matrices $H(v)$, one
for each $v\in V$.  Each $H(v)$ is a $\{0,1\}\times\{0,1\}$ rank-1 matrix of the form
\be H(v)=\one \cdot h_v, \label{mmylss} \ee
where $h_v=(h_v(0),h_v(1))$ with ($h_v(1)=1-h_v(0)$) is a row vector whose components are the probability of
non-occupation/occupation of the heat-bath at~$v$; i.e., $H(v)$ replaces the value~$\sigma(v)$ by 
a $\{0,1\}$ value drawn from the distribution~$h_v$.

We collect the infinitesimal dynamics in two $\{0,1\}^V \times \{0,1\}^V$ matrices.
The first matrix represents solely the internal dynamics:
\begin{equation}\label{Sin}
 \cSin := \sum_{\{u,v\}\in E} \eta_{\{u,v\}} \cdot S(\{u,v\}) \otimes I(V\backslash \{u,v\}),
\end{equation}
where $I(V\backslash \{u,v\})$ is the identity operator on the space $\{0,1\}^{V\backslash \{u,v\}}$. 
Observe that the internal dynamics have a conserved quantity, namely the total occupancy $\sum_{v\in V} \sigma_t(v)$. 

The second matrix represents the external, heat bath interactions:
\begin{align} \cH &:= \sum_{v \in V'} \xi_{v} \cdot H(v) \otimes I({V\backslash \{v\}})
\label{Hh}. \end{align}
Now the rate matrix \be \cS:=\cSin+\cH \label{cSdef} \ee
defines a continuous-time dynamics (a Markov chain, in general non-reversible) via the infinitesimal generator $\cL := \cS-cI$,
where $I$ is the identity matrix on $\{0,1\}^V$ and $c=\sum_{\{u,v\}\in E} \eta_{\{u,v\}}+\sum_{v\in V'} \xi_{v}$.  
Note that the internal component~$\cSin$, and the heat bath rates~$\xi_v$ are
fixed for a given graph~$G$, while the heat bath temperatures~$\cH$ may be viewed
as variable boundary conditions.

%%%%%%%%%%%%%%%%%%%%%%%%%%%%%%%%%%%%%%%%%%%%%%%%

\subsection{Results}\label{subs:results}
Our first main result is a tight (up to logarithmic factors) bound on the rate of convergence to stationarity
of the SEP on any graph with arbitrary heat baths.  
To state this result, we define the mixing time of the SEP as follows.  Let $p^{(t)}$ denote the probability distribution
of the configuration~$\sigma_t$ at time~$t$, and $\mu$ the steady-state distribution.  Then we define
\begin{equation}
   \tmix(\varepsilon) := \max_{\sigma(0)}\inf\{t:\TVD{p^{(t)}}{\mu}\le\varepsilon\},
\end{equation}
where $\Vert \cdot\Vert_{\rm TV}$ denotes total variation distance.  Sometimes, slighly informally, we shall 
refer to the {\it mixing time\/}~$\tmix$ as $\tmix(\varepsilon)$ for a small, fixed constant~$\varepsilon$.

Our bound on the mixing time will be phrased in terms of the maximum expected hitting time to a heat bath
of a continuous time random walk that is a single-particle version of the SEP.
More precisely, consider the
following random walk on~$G$, driven by the edge weights $\eta_{\{u,v\}}$ and heat bath weights~$\xi_v$.
The walk is executed by a single particle, starting at some vertex~$v_0$.  Whenever an edge or heat
bath clock rings, if the corresponding edge is incident on the current position of the particle, the particle
moves along that edge; if  the edge is a heat bath edge, the particle is absorbed by the corresponding
bath.  Define $\hit$ to be the maximum expected hitting time from any initial vertex~$v_0$ to the set of 
heat baths.  Our first result shows that~$\hit$ is, up to factors of $\log|V|$,  a tight bound on the mixing time
for any SEP with arbitrary heat baths.

\begin{theorem}\label{thm:mixing}
For any SEP with arbitrary heat baths on an $n$-vertex connected graph, the mixing time is $O(\hit\log n)$
and is also $\Omega(\hit/\log n)$.  More precisely:
\begin{enumerate}
\item[(a)] $\tmix(\varepsilon)\le 2\log_2(n\varepsilon^{-1})\hit$ for any $\varepsilon\in(0,1)$;
\item[(b)] $\tmix(1/40) > \hit/a\ln n$ for some universal constant~$a>0$.
\end{enumerate}
\end{theorem}

We prove this theorem in Section~\ref{sec:mixing}.  The upper bound is a straightforward coupling argument, 
based on the idea that mixing occurs once the influence of the heat baths has reached every vertex.  The lower
bound is more delicate due to the need to construct a suitable test function.

As mentioned earlier, in the non-equilibrium setting the steady-state distribution on general graphs is
hard to describe due to the absence of reversibility.  However, using a generalization of Liggett's duality
theory~\cite[Chap.~VIII, Thm.~1.1]{Liggett85}, we can describe the joint
distribution at time~$t$ of the exclusion process at any $k$ vertices in terms of the absorption 
probabilities of an exchange process with $k$ particles: see Section~\ref{sec:correlations} for details. 
This generalization is apparently folklore but to the best of our knowledge has not been explicitly written down anywhere;  
however, see~\cite{KipnisMP82} for an analogous application of duality in a different exchange process.  
Combining duality with our upper bound on mixing time, we obtain a polynomial time 
algorithm for computing any such $k$-wise joint distribution for any fixed~$k$, as stated in our next result.
The algorithm works by computing the time evolution of $\hit\log(n/\varepsilon)$ steps of the $k$-particle
process, which turns out to be a Markov chain on a much smaller state space of size only~$(n+2)^k$.
\begin{theorem}\label{thm:kwise}
For any SEP with arbitrary heat baths, there is an algorithm running in time $n^{O(k)}(\log \hit +\log\log\varepsilon^{-1})$ that
computes the joint occupation distribution of any $k$ vertices in~$G$ within variation distance~$\varepsilon$, 
for any $\varepsilon\in(0,1)$.
\end{theorem}

Finally we turn to the spectral result outlined in the introduction.
Given an SEP defined by the rate matrix~\eqref{cSdef} above, 
with heat baths $H(v)=\one\cdot h_v$,
consider a second SEP with the same internal dynamics~$\cSin$ and heat bath interactions~$\xi_v$
but different heat baths~$H'(v)=\one\cdot h'_v$; its corresponding rate matrix is then
\begin{equation}\label{cSdef'}
 \cS'  = \cSin+\cH' = \cS +\cH' -\cH.
\end{equation}
The following theorem says that the spectrum of any SEP is invariant under changes of the
heat bath temperatures.  Let $\spec$ denote the multiset of eigenvalues of a matrix.

\begin{theorem} \label{septhm}
For any SEP~$\cS$ on a connected graph with $\sum_v\xi_v>0$, 
the spectrum is unaffected by the driving heat baths, 
i.e., for any $\cS'$ as in~\eqref{cSdef'}, we have $\spec(\cS')=\spec(\cS)$. 
\end{theorem}

To interpret Theorem~\ref{septhm}, let us set $h_v = \hrf$ for all $v\in V$,
where $\hrf$ is a strictly positive probability vector; we call~$\Hrf$, the global heat bath matrix
defined by $\hrf$, the ``reference" heat bath.  This ensures that the dynamics is reversible w.r.t.\ 
the product distribution $\murf(\sigma)=\prod_{v\in V}\hrf(\sigma(v))$.
Since our connectedness assumption also ensures that the dynamics is irreducible (see
Proposition~\ref{prop:basic}), this implies that it is ergodic with steady-state distribution~$\murf$. 
Moreover, the rate matrix $\cS=\cSin+\Hrf$ corresponding to the reference boundary
condition has real spectrum. 
 In particular, the Laplacian $\cL=\cS - cI$ 
corresponding to the reference boundary condition has principal eigenvalue~0 of multiplicity~1,
and the rest of the spectrum is negative real.  Theorem~\ref{septhm} implies that the spectrum
of any alternative dynamics~$\cL'=\cS'-cI$ with the same~$\cSin$ and heat bath interactions~$\xi_v$,
but \emph{arbitrary} heat bath temperatures~$\cH'$ 
(not necessarily uniform over $v\in V$) is the same as that of~$\cL$ (and hence real), even though $\cL'$ is
in general not reversible. Of course the steady-state distribution of $\cL'$ depends on~$\cH'$.

We observe also that, although the non-equilibrium dynamics shares a spectrum with the 
reversible dynamics~$\cL$, the spectrum is not as informative for general~$\cL'$: for example, 
the rate of convergence to stationarity in the reversible case is determined
by the spectrum (and bounded in terms of the spectral gap), but this
connection is lost in the absence of reversibility, where it is known that the convergence time of a Markov chain can be dramatically
longer than predicted by the spectrum for the reversible case~\cite{MehtaS20}.
We remark also that our proof of Theorem~\ref{septhm} uses a method based on ladder operators and Lax pairs
that in principle extends beyond the SEP to other boundary-driven dynamics.  In a later version of the
paper we will give examples of other settings in which spectral invariance holds, not necessarily for 
all heat bath temperatures but for a suitable model-dependent equivalence class of temperatures.

\section{Mixing time bounds}\label{sec:mixing}
In this section, we will prove the upper and lower bounds on mixing time summarized in Theorem~\ref{thm:mixing}.
In preparation for this, let us first confirm that the steady-state distribution is well defined.
\begin{proposition}\label{prop:basic}
For any SEP on a connected graph satisfying the non-triviality condition on the heat baths stated earlier,
the steady-state distribution~$\mu$ exists and is unique.  Moreover, if all heat baths are identical
(i.e., the 2-point probability distribution at every vertex~$v$ with $\xi_v>0$ is $h_v=(h(0),h(1))$, 
with $h(0)=1-h(1)\in (0,1)$), then the SEP is reversible w.r.t.\ the steady-state product distribution 
$\mu(\sigma)=\prod_{v\in V} h(\sigma(v))$.
\end{proposition}

\begin{proof}
The SEP induces a continuous-time Markov chain on $\Omega:=\{0,1\}^V$, the set of possible
occupation vectors on the vertices of~$G$.  This chain will be ergodic if it is irreducible on~$\Omega$.

To see that it is irreducible, we show how to create any occupation vector~$\sigma$ on~$G$ from
any initial configuration. Let $u$ be a vertex with a heat bath satisfying $h_u(0) > 0$.  (Such a 
vertex exists by assumption.)   Construct a spanning tree rooted at~$u$.  We can first
set all values to~0 level-by-level in the tree as follows.  By exchanging with
the heat bath if necessary, ensure that $u$ has value~0, then perform a sequence of edge
exchanges on the path from~$u$ to a leaf to move this 0 value down to the leaf.  
Once all leaves are set to~0, we use the same procedure to set all values at the level
above the leaves to~0, and so on up to the root~$u$.

Then we select a vertex~$v$ with a heat bath satisfying $h_v(1) > 0$.  (Again, such a vertex
is guaranteed to exist; note that $v$ may be the same as~$u$.)  We construct a spanning tree
rooted at~$v$ and perform the same procedure as above to change the value of any desired
subset of vertices to~1, by repeatedly exchanging with the heat bath at~$v$ to set its value to~1
and then moving this value to the desired site in the tree level-by-level.

Finally, to see the claim for the case of identical heat baths, note that the SEP dynamics
restricted to non-heat-bath exchanges is reversible w.r.t.\ the product measure 
$\mu(\sigma)=\prod_{v\in V} h(\sigma(v))$, since these exchanges conserve the number
of occupied sites.  And reversibility for the heat bath exchanges is also immediate.
\end{proof}

We now proceed to the proof of Theorem~\ref{thm:mixing}, which quantifies the rate of convergence of the
dynamics to the steady-state distribution~$\mu$.  We will prove the upper and lower bounds, parts~(a) and~(b), separately.
 
\begin{proofof}{Theorem~\ref{thm:mixing}(a)}
Let $G=(V,E)$ be the underlying connected graph.
Recall that the states of the SEP can be represented as configurations $\sigma\in \{0,1\}^V$, where
$\sigma(v)=1$ (respectively, $\sigma(v)=0$) denotes the fact that vertex~$v$ is occupied (respectively,
unoccupied).  It will be convenient in what follows to adopt the equivalent view in which
there is one particle at every vertex $v\in V$, carrying the {\it label}~$\sigma(v)$.  When the clock on edge
$\{u,v\}$ rings, the particles at~$u,v$ (along with their labels) switch positions.  Heat bath exchanges are
implemented similarly, except that the label of a particle at the bath at~$v$ is always chosen from the
distribution specified by~$h_v$.

Let $\sigma_t,\tilde\sigma_t$ denote two copies of the dynamics, started in arbitrary initial configuration $\sigma_0, \tilde\sigma_0$.
We construct a coupling between these dynamics as follows:
\begin{itemize}
\item both $\sigma_t$ and $\tilde\sigma_t$ share the same Poisson clocks, so that an exchange on edge~$e$,
or with a heat bath at a vertex~$v$, occurs in~$\sigma_t$ iff the same exchange occurs in~$\tilde\sigma_t$;
\item if the exchange is with a heat bath at vertex~$v\in V'$, choose the same 0-1 label for~$v$ (a sample
from the heat bath distribution~$h_v$, independent of the history of the process) in both $\sigma_t$ and $\tilde\sigma_t$.
\end{itemize}
For any given initial configurations $\sigma_0,\tilde\sigma_0$, we define the {\it coupling time}~$\tcoup$ by $$
   \tcoup = \max_{\sigma_0,\tilde\sigma_0} \inf\{t: \sigma_t=\tilde\sigma_t\}.  $$
By standard machinery (see, e.g., \cite[Chap.~5]{Peresbook}), the mixing time of the chain is
bounded above by the tail of the coupling time, i.e.,
\begin{equation}\label{eqn:peres}
   \tmix(\varepsilon) \le \inf\{t:\Pr[\tcoup\ge t]\le\varepsilon\}.
\end{equation}

To analyze the coupling, it is helpful to attach to each particle in the current configuration a {\it color\/} 
from the set $\{\rm{red}, \rm{black}\}$.  Initially all particles are black.
Particles retain their colors as they move around the graph via exchanges, except that
whenever a particle is introduced by an exchange with a heat bath it is red. 
Note that the number of red particles in the configuration is non-decreasing (colors never
switch from red to black), and moreover the number of red particles increases by~1 each
time a black particle is exchanged with a heat bath.
Note also that the coupling ensures that the labels of all red
particles in the configurations $\sigma_t,\tilde\sigma_t$ agree.  Hence the coupling time is bounded above by the
time until all particles become red.

This time is easy to analyze.  Consider a particular black particle, initially at vertex~$v$.  This
particle will be replaced by a red particle as soon as it performs an exchange with a heat bath.  
Thus the time until this happens is just the hitting time from~$v$ to the set of heat
baths in the single-particle random walk defined earlier (where, in that walk, the particle is
absorbed).  Let $\hit$ denote the maximum expected hitting time from any vertex~$v$ to
the set of heat baths.
By Markov's inequality, and taking $k$ independent trials, we have that, for any given particle, $$
   \Pr[\hbox{\rm particle has not become red after $2k\hit$ steps}] \le 2^{-k}.  $$
Taking a union bound over all $n=|V|$ particles gives $$
   \Pr[\hbox{\rm any black particle remains after $2k\hit$ steps}] \le n2^{-k}.  $$
Setting $k = \log_2(n\varepsilon^{-1})$ for constant~$c$, we see that the coupling time exceeds
$2\hit\log_2(n\varepsilon^{-1})$ with probability at most~$\varepsilon$.  This completes the proof via~\eqref{eqn:peres}.
\end{proofof}

We now prove the complementary lower bound on the mixing time in Theorem~\ref{thm:mixing}(b).

\begin{proofof}{Theorem~\ref{thm:mixing}(b)}
The proof hinges on the following main lemma, in which we adopt the same labeled particle
view and red/black coloring convention as in the proof of part~(a) of the theorem.
\begin{lemma}\label{lem:lowerbd}
There exists a subset $S\subseteq V$ such that, at time
$t=\hit/(a\ln n)$ for some universal constant $a>0$, the expected fraction of black particles in~$S$ in the
configuration~$\sigma(t)$ is at least~$\frac{1}{20}$.
\end{lemma}

We proceed with the proof of the theorem and return to prove the lemma later.

Let $\mu$ denote the steady-state distribution of the SEP, and $p^{(t)}$ the distribution at time~$t$,
where the initial configuration is either $\sigma_0 = {\bf 1}$ or $\sigma_0 = {\bf 0}$, as will be specified below.
For any configuration~$\sigma$, let $f(\sigma)$ denote the fraction of particles in~$S$ with label~1
(equivalently, the fraction of vertices in~$S$ that are occupied).
We claim that, for a suitable choice of~$\sigma_0$, and $t=\hit/(a\ln n)$, we have
\begin{equation}\label{eq:expectn}
         |{\rm E}_{p^{(t)}} f - {\rm E}_{\mu} f | \ge \frac{1}{40},
\end{equation}         
where ${\rm E}_p f$ denotes the expectation of~$f$ wrt the distribution~$p$.  From~\eqref{eq:expectn}
and the fact that $f$ takes values in~$[0,1]$,
we immediately get the claimed bound on variation distance via $$
   \Vert p^{(t)} - \mu \Vert_{\rm TV} \ge  |{\rm E}_{p^{(t)}} f - {\rm E}_{\mu} f | \ge \frac{1}{40}.  $$
   
To establish~\eqref{eq:expectn}, write 
\begin{equation}\label{eq:convex}
   {\rm E}_{p^{(t)}} f = \begin{cases} F + G & \hbox{\rm if {$\sigma_0={{\bf 1}}$};}\cr
                                                          G & \hbox{\rm if {$\sigma_0={{\bf 0}}$}},\cr
   \end{cases}
\end{equation}
where, at time~$t$, $F$ is the expected fraction of particles in~$S$ that are black and $G$ is the expected fraction
of particles in~$S$ that are red and have label~1.  Lemma~\ref{lem:lowerbd} implies that $F\ge\frac{1}{20}$.
Since under these two choices of the initial configuration~$\sigma_0$ we obtain two values of ${\rm E}_{p^{(t)}} f$ that 
differ by~$F$, at least one of them differs from ${\rm E}_{\mu} f$ by at least $F/2\ge\frac{1}{40}$.  
Using this value for~$\sigma_0$ completes the proof of Theorem~\ref{thm:mixing}(b).
\end{proofof}

We now supply the proof of the lemma above.

\begin{proofof}{Lemma~\ref{lem:lowerbd}}
For any vertex $v\in V$, let $\tau(v)$ denote the expected hitting time in the single-particle random walk
from~$v$ to the set of heat baths. 
Thus $\hit=\max_v \tau(v)$.  For non-negative integers~$k$, define $\alpha_k:=\frac{k}{\log_c n}$, where $c>2$
is a constant to be specified later, and consider the increasing sequence of subsets $S_k\subseteq V$ defined by $$
   S_k := \left\{v\in V: \tau(v) > (1-\alpha_k)\hit\right\},\qquad 1\le k\le 1+\lfloor\log_c n\rfloor.  $$
Let $K$ be the smallest~$k$ for which $\frac{|S_{k+1}|}{|S_k|} \le c\,$; such a~$K$ must exist
since $|V|=n$.  We will take $S_{K+1}$ to be the set~$S$ whose existence is claimed in the Lemma, 
and denote by~$F$ the expected fraction of black particles in~$S$. Our goal is to show that $F\ge\frac{1}{20}$.

For $2\le k\le K+1$, define $T_k:=S_k\setminus S_{k-1}$, and $T_1:=S_1$.  Note that for any vertex $v\in T_k$ we have $$
    (1-\alpha_k)\hit < \tau(v) \le (1-\alpha_{k-1})\hit.  $$
On the other hand, for $v\in T_K$ we also have $$
    \tau(v) \le t + (1-p_0)(1-\alpha_{K+1})\hit + \sum_{r=0}^{K} q_r(1-\alpha_{K-r})\hit,   $$
where $p_0$ is the probability that the particle that was initially at~$v$, is black and lies inside $S=S_{K+1}$ at time~$t$,
and $q_r$ is the probability that the particle is black and lies in~$T_{K-r+1}$ at time~$t$.  Note that $p_0=\sum_{r=0}^K q_r $.
Combining the above two bounds yields $$
    (1-\alpha_K)\hit \le t + (1-p_0)(1-\alpha_{K+1})\hit + \sum_{r=0}^{K} q_r(1-\alpha_{K-r})\hit.  $$
Plugging in $t=\hit/(a\ln n)$ and using the fact that $p_0 = \sum_r q_r$, this becomes $$
    -\alpha_K \le \frac{1}{a\ln n}  -(1-p_0)\alpha_{K+1} - \sum_{r=0}^{K} q_r\alpha_{K-r}.  $$
Recalling that $\alpha_k=\frac{k}{\log_c n}$, we may rewrite this as $$
    -K \le \delta -(K+1)(1-p_0) - \sum_{r=0}^{K} q_r(K-r), $$
where $\delta:=(a\ln c)^{-1}$.  We assume in what follows that $K\ge 2$; if not, then a simpler version of 
the calculation below yields a better lower bound on~$F$.  Using $p_0 = \sum_r q_r$ again and rearranging gives $$
     p_0 \ge 1-\delta - \sum_{r=0}^{K} rq_r = 1 - \delta - q_1 - \sum_{r=2}^K rq_r = 1 - \delta -p_0 + q_0 - \sum_{r=2}^K (r-1)q_r,  $$
which, discarding the~$q_0$ on the RHS, implies
\begin{equation}\label{eq:pbound}
     p_0 \ge \frac{1}{2}\Bigl(1-\delta - \sum_{r=2}^K (r-1)q_r\Bigr) =\frac{1}{2}\Bigl(1-\delta - \sum_{r=2}^K p_r\Bigr),
\end{equation}     
where we have introduced the notation $p_r:=\sum_{j=r}^K q_j$ for the probability that the particle is black
and lies in~$S_{K-r+1}$ at time~$t$.  Note that this is consistent with our earlier definition of~$p_0$.
   
Consider the particles that are in $T_K$ at time $0$. Let $\hat p_r$ be the expected fraction of these particles
that at time~$t$ remain black and inside~$S_{K-r+1}$.  Note that our goal is a lower bound on~$\hat p_0$, which 
is the expected fraction of black particles remaining in~$S_{K+1}$.

Applying~\eqref{eq:pbound} to $v$ uniformly sampled in $T_K$ gives
\begin{equation}\label{eq:bbound}
    \hat p_0 \ge \frac{1}{2}\Bigl(1-\delta - \sum_{r=2}^K \hat p_r\Bigr).
\end{equation}    

Now we observe that, since each vertex contains only one particle at any time,
$\hat p_r \le \frac{|S_{K-r+1}|}{|T_K|}$.  Moreover, by our choice of~$K$,
$|S_{K-r+1}|\leq |S_K|c^{-(r-1)}$ for $r\ge 2$.  Also, $|T_K| = |S_K|-|S_{K-1}|\geq |S_K|(c-1)/c$.
Combining these observations we see that, for any $r\ge 2$,
$$
   \sum_{r=2}^K \hat p_r  \leq \frac{c}{c-1}\sum_{r=2}^K  c^{-(r-1)} 
           \leq \frac{c}{c-1} \times \frac{1}{c-1} = \frac{c}{(c-1)^2}.  $$

Plugging this bound into~\eqref{eq:bbound} yields
\begin{equation}\label{eq:bbound2}
   \hat p_0 \ge \frac{1}{2}(1-\delta -c(c-1)^{-2}).
\end{equation}   
Noting from our definition of~$K$ that 
$\frac{|T_K|}{|S|} = \frac{|S_K|-|S_{K-1}|}{|S_{K+1}|} \ge \frac{(1-1/c)|S_K|}{|S_{K+1}|} \ge \frac{c-1}{c^2}$,
and combining this bound with that on~$\hat p_0$ in~\eqref{eq:bbound2} yields the following lower bound on the desired
fraction~$F$: $$
    F\ge \hat p_0\frac{|T_K|}{|S|} \ge \frac{1}{2}(1-\delta -c(c-1)^{-2})\times \frac{c-1}{c^2}.  $$
Finally, setting $c=5$ and $\delta=\frac{1}{16}$ (which is achieved by setting the constant $a=16/\ln 5\approx 10$) 
results in $F\ge\frac{1}{20}$, as desired.
\end{proofof}

{\bf Examples:} {\it (i)}.\ \
In the classical 1-dimensional SEP, where
$G$ is the path with $n$ vertices labeled $1,\ldots,n$ and just two
heat baths attached to the endpoints~1 and~$n$, respectively, and all exchange rates
$\eta_{uv}=1$ and $\xi_1=\xi_n=1$, then clearly the maximum expected hitting time to
the heat baths is $\hit=\Theta(n^2)$, and hence Theorem~\ref{thm:mixing} gives upper
and lower bounds $\Omega(n^2/\log n)\le \tmix\le O(n^2\log n)$.  Thus for this special case
our upper bound matches asympototically the tight bound obtained in~\cite{Nestoridi}, 
while our lower bound is off by a factor of $\log^2 n$.

{\it (ii)}.\ \ More generally, let $G$ be a rectangular region of size $n_1\times\cdots\times n_d$ in the
$d$-dimensional lattice~${\Bbb Z}^d$, with $n_1\le\ldots\le n_d$, and suppose that all vertices on the external faces are 
heat baths.  If all edge and heat bath exchange rates are again~1, and $d$ is constant, then we get $\hit=\Theta(n_1^2)$
and therefore $\Omega(n_1^2/\log n_d)\le \tmix\le O(n_1^2\log n_d)$.
In particular, if $n_i=\bar n$ for all~$i$ (a cube) then $\Omega(\bar n^2/\log \bar n)\le \tmix\le O(\bar n^2\log \bar n)$.
The upper bound here agrees with the more precise bound obtained in~\cite{Salez23} for the special case where the 
heat bath temperatures are all equal. 
We stress again that these bounds hold no matter what the
heat bath temperatures are.  
We note also that analogous bounds hold even when, say, the heat baths reside
only on two parallel faces of the cube; this is a finite version of the scenario whose steady-state distribution was studied, for heat baths of uniform temperatures on both faces, by Spohn~\cite{Spohn83} in the statistical physics literature. 

{\it (iii)}.\ \ 
As a final example, we again consider the $d$-dimensional cube of side-length~$\bar n$, but now with two heat baths,
situated at opposite corners $(0,\ldots,0)$ and $(\bar n,\ldots,\bar n)$, as discussed in~\cite{Bodineauetal}.  In this case
the hitting time $\hit$ is $\Theta(\bar n^d)$ for all $d\ge 3$ and $\Theta(\bar n^2\log \bar n)$ for $d=2$ (see, e.g., 
\cite{AldousBrown,DoyleSnell}).  Once again, Theorem~\ref{thm:mixing}
implies that the mixing time is bounded above and below by these values up to a correction factor of $O(\log \bar n)$.

\section{Correlations in the steady-state distribution}\label{sec:correlations}
In this section we discuss the steady-state distribution~$\mu$ of the SEP on arbitrary graphs.
As discussed earlier, due to the lack of reversibility of the SEP dynamics (except in the
case of identical heat baths), $\mu$~is in general hard to determine.  A closed form for
the joint occupation probabilities of any set of sites
is available in the special case of the 1-dimensional SEP~\cite{Kingman69}, but for
more general graphs useful expressions even for pairwise correlations are elusive.

Our goal here is to prove Theorem~\ref{thm:kwise} stated earlier, which gives a polynomial
time algorithm for computing the joint occupation probabilities of any $k$ sites for any
fixed~$k$ in an arbitrary connected graph with arbitrary heat baths.  
The algorithm follows almost immediately from our Theorem~\ref{thm:mixing}(a)
combined with a powerful duality property of the SEP.  This duality property for the
standard SEP (without heat baths) is a classical result due to Liggett~\cite{Liggett85}.
Here we extend that duality property to include arbitrary heat baths.  We believe that this 
extension is folklore but are not aware of a published version; however, Kipnis 
{\it et al.}~\cite{KipnisMP82} derive an analogous result for a different exchange process.

To state the duality property, we again recast the SEP in terms of an exchange process
involving labeled particles, as in the previous section.  However, we now extend the label
set to $\{0,1\}\cup V$, i.e., we view the configuration at time~$t$ as a vector 
$\sigma_t\in (\{0,1\}\cup V)^V$.  The process is initialized with $\sigma_0(v)=v$ for all $v\in V$.
When the Poisson clock on an edge $\{u,v\}$ rings, the labeled particles at $u$ and $v$ are exchanged.  
When a heat bath clock at vertex~$v$ rings, the particle at $v$ is exchanged for one with 
label in $\{0,1\}$, sampled (independently of all prior events) from the distribution~$h_v$.
Thus as the process evolves, $\sigma_t(v)$ will be either the $\{0,1\}$ value that the particle
currently at~$v$ received at its {\it last\/} exchange with a heat bath, or 
the initial position $u\in V$ of that particle if no such exchange has yet occurred.

We note that, as $t\to\infty$, we will have $\sigma_t(v)\in \{0,1\}$ for
all~$v$ a.s.\ since the hitting times from all vertices to the set of heat baths
are finite a.s. Moreover, the distribution of $\sigma_t$ will approach the steady-state
distribution~$\mu$ as $t\to\infty$.
Note also that we can easily recover the original representation of the SEP from this
representation by replacing any remaining vertex label~$v$ by the initial
$\{0,1\}$ value at vertex~$v$ in the original SEP.

Next we define the associated {\it dual process}.  
As in the above process, the configuration of the dual process at time~$t$ is
a vector $\hat\sigma_t \in (\{0,1\}\cup V)^V$ that assigns values in
$\{0,1\}\cup V$ to vertices.  However, the meaning of these values is
now different.  Specifically, $\hat\sigma_t(v)$ records either the current position $u\in V$
of the particle that started at~$v$ if that particle has not yet undergone
an exchange with a heat bath, or the $\{0,1\}$ value that that particle 
acquired in its first heat bath exchange.  
As soon as $\hat\sigma_t(v)$ acquires a $\{0,1\}$ value, that value remains
for all future time: thus we can think of particles as being absorbed upon their
first exchange with a heat bath.  As in the previous process,
the initial configuration is $\hat\sigma_0(v)=v\;\forall v$.
Note that, as for the SEP process defined above, as $t\to\infty$ we have 
$\hat\sigma_t(v)\in \{0,1\}$ a.s.\ for all $v\in V$.

The following claim is almost immediate from the above definitions:
\begin{proposition} \label{prop:duality}
For any subset of sites $\{v_1,\ldots,v_k\}\subset V$, the SEP process~$\sigma_t$ 
and its dual~$\hat\sigma_t$ are related as follows:
\begin{equation*}
   \Pr[\sigma_t(v_1,\ldots,v_k)=(z_1,\ldots,z_k)] = 
   \Pr[\hat\sigma_t(v_1,\ldots,v_k)=(z_1,\ldots,z_k)]\quad\forall (z_1,\ldots,z_k)\in(\{0,1\}\cup V)^k.  
\end{equation*}   
\end{proposition}
\begin{proof} Couple the evolutions of $\sigma_t,\hat\sigma_t$ as follows: let the
sequence of edge exchanges up to time~$t$ in $\sigma_t$ (including heat bath exchanges)
be denoted $e_1,\ldots,e_m$, occurring at times $t_1,\ldots,t_m$, and augment this
sequence with the heat bath samples chosen at each heat bath exchange.  Define
the dual evolution $\hat\sigma_t$ by using the reversed sequence $e_m,\ldots,e_1$
at times $t_1,\ldots,t_m$, along
with the same heat bath samples at corresponding steps.  This coupling pairs up
the particle in~$\sigma_t$ that ends at position~$v_i$ after $t$ steps with the
particle in $\hat\sigma_t$ that starts at~$v_i$.  By definition of the processes, if one of these
particles has a $\{0,1\}$ label then the other has the same label; and if not, then both 
carry the same label~$u$ corresponding to the common initial position of the first
particle and final position of the second particle.
\end{proof}

Note that, despite this duality, the limiting behaviors of $\sigma_t$ and $\hat\sigma_t$ are different. In
$\hat\sigma_t$, for all $\{v_1,\ldots,v_k\}\subset V$ there is a.s.\ a time $t$ such that
$\hat\sigma_t(v_1,\ldots,v_k)\in \{0,1\}^k$ and for all $t'\geq t$, 
$\hat\sigma_{t'}(v_1,\ldots,v_k)=\hat\sigma_t(v_1,\ldots,v_k)$. On the other hand, in $\sigma_t$, 
for all $\{v_1,\ldots,v_k\}\subset V$ the distribution of $\sigma_t(v_1,\ldots,v_k)$ converges to a
distribution supported on $\{0,1\}^k$.  The content of Proposition~\ref{prop:duality} is that this limit distribution
is the same as that of the fixed point of $\hat\sigma_t$.

\begin{remark}
Proposition~\ref{prop:duality} includes Liggett's Theorem 1.1 in~\cite[Chapter VIII]{Liggett85} as a special case,
where there are no heat baths.  In that case all particle labels are in~$V$, and the proposition
says that the sequence of particles that arrive at
$v_1,\ldots,v_k$ at time~$t$ under the SEP is the same as the sequence of
final locations of particles starting at $v_1,\ldots,v_k$ under the dual process.
This gives Liggett's theorem, once we identify the labels of the particles with their
occupied/unoccupied status in the initial configuration.
\end{remark}

We are now in a position to prove Theorem~\ref{thm:kwise} stated earlier.
\begin{proofof}{Theorem~\ref{thm:kwise}}
Fix a set of $k$ sites $v_1,\ldots,v_k$ and any desired $\varepsilon\in(0,1)$.  
From the proof of Theorem~\ref{thm:mixing}(a) 
we see that, for any initial configuration~$\sigma_0$, the distribution of the configuration~$\sigma_t$
of the SEP at time $t=\hit\log_2(n\varepsilon^{-1})$ is within total variation distance~$\varepsilon$ of
the steady-state distribution~$\mu$.  This implies that the same holds for the distribution
of $\sigma_t(v_1,\ldots,v_k)$ in the alternative representation of Proposition~\ref{prop:duality},
since red particles (in the language of the proof of Theorem~\ref{thm:mixing}(a)) are precisely
the particles with $\{0,1\}$ labels in Proposition~\ref{prop:duality}.
By that proposition, this distribution is identical to that of $\hat\sigma_t(v_1,\ldots,v_k)$.
Now the crucial observation is that the dual process~$\hat\sigma_t$, restricted to
the subset $\{v_1,\ldots,v_k\}$, is also a Markov chain whose state space has size $(n+2)^k$,
since it just involves tracking the motions of $k$ particles starting at $v_1,\ldots,v_k$.
(Note that, crucially, the same is {\it not\/} true of the restriction of the original SEP~$\sigma_t$.)
Thus we can compute the distribution of this restriction at time~$t$ by repeated squaring
of its transition matrix in time $n^{O(k)}\log t$.  Setting $t=\hit\log_2(n/\varepsilon)$
completes the proof.
\end{proofof}

\section{Spectrum invariance}\label{sec:spectral}
In this section we prove our spectral result, Theorem~\ref{septhm}.
We begin by describing a general method based on ladder operators and Lax pairs.

Recall that our goal is to compare the spectra of two rate matrices $\cS,\cS'$ that share the
same internal dynamics~$\cSin$ and heat bath interactions~$\xi_v$
but have different heat bath temperatures encoded in the matrices $\cH,\cH'$, respectively.
To this end, define \be A(v)=H'(v)-H(v) = \one \cdot a_v, \label{avmless} \ee
where $a_v=h'_v-h_v$, and form
\begin{align} 
\cA &= \cH' - \cH
= \sum_{v \in V} \xi_{v} \cdot (\one \cdot a_v) \otimes I(V\backslash \{v\}).
 \label{A} \end{align}
Note that all row sums of~$\cA$ are zero.
Since $\cS' = \cS+\cA$, our goal is to show that $\cS$ and $\cS+\cA$ are co-spectral. 
The main tool is this:

\begin{lemma} If there exists a matrix $\cT$ such that 
\begin{align} [\cS,\cT]&=0 \label{STcomm}; \\
 [\cA,\cT]&=c \cA \quad \text{ for real $c \neq 0$. } \label{ATcomm} \end{align}
 then $\spec(\cS+\cA)=\spec(\cS)$. \label{STA} \end{lemma}
Property~\ref{ATcomm} states that $\cA$ is a {\it ladder operator\/} for $\cT$ with increment $c$
(so-called because if $v$ is a left eigenvector of $\cT$ with eigenvalue $\lambda$, then 
$(v \cA) \cT = v (\cT \cA +  [\cA,\cT]) =(\lambda+c) v \cA$, i.e., $v \cA$ is a left eigenvector with eigenvalue $\lambda+c$.) 
Since $\cS,\cA,c$ are real, $\cT$ may be assumed w.l.o.g.\ real.
 
\begin{proof} Consider the one-parameter family of matrices $M(r):=e^{(r/c) \cT}(\cS+\cA)e^{-(r/c) \cT}$ ($r$ real). 
Clearly these matrices are co-spectral. We claim that
\begin{equation}\label{eq:clm}
M(r)=\cS+e^{-r}\cA.
\end{equation}
This will complete the proof of the Lemma, since it immediately implies that $\lim_{r \to \infty} M(r)=\cS$,
which in turn implies the Lemma because the spectrum $\spec(M)$ of a square matrix $M$ is continuous 
in~$M$ (in Hausdorff distance)~\cite[VIII.3]{Bhatia97}.

To verify~\eqref{eq:clm}, note first that, due to property~\eqref{STcomm}, 
\be e^{(r/c)\cT}\cS e^{-(r/c) \cT}=\cS. \label{tst} \ee
Next, for square matrices $C,D$, define the iterated commutators
\begin{align} [^0C,D]&=D; \\ [^nC,D]& =[C,[^{n-1}C,D]] \quad \text{ for } n\geq 1. \end{align} 
The Campbell-Baker-Hausdorff formula~\cite{Sternberg04,FSchur1889} gives
\begin{align}\label{eq:CBH} e^C D e^{-C} = \sum_{n=0}^\infty \frac{1}{n!} [^nC,D]. \end{align}
Now property~\eqref{ATcomm} implies that $[^n\cT,\cA]=(-c)^n \cA$ for $n\geq 0$,
so $[^n(r/c) \cT,\cA]=(-r)^n \cA$.  Thus, applying~\eqref{eq:CBH} to $e^{(r/c) \cT}\cA e^{-(r/c) \cT}$,
we obtain
\begin{align} e^{(r/c)\cT}\cA e^{-(r/c)\cT} &= e^{-r} \cA, \label{tat} \end{align}
which together with~\eqref{tst} completes the verification of~\eqref{eq:clm} and hence the proof of the Lemma.
\end{proof}

\begin{remark} An alternative way to argue this Lemma would be to \emph{define} $M(r)$ by~\eqref{eq:clm}, and then to show that $M(r)$ and~$\cT$
form a {\it Lax pair}~\cite{Lax68}, since properties~\eqref{STcomm} and~\eqref{ATcomm} ensure that $\frac{d}{dr}M(r) =\frac{1}{c} [\cT,M(r)]$; this
in turn implies the claimed spectral invariance property.
\end{remark}

\begin{remark}
One cannot prove Lemma~\ref{STA} by showing (as one might na\"ively have hoped) that $\cS+\cA$ and $\cS$ are similar matrices. 
It is even possible for $\cS$ to be reversible (and thus diagonalizable), while $\cS+\cA$ is non-diagonalizable.
A simple concrete example is a three-site SEP on a path, with $V=\{v_1,v_2,v_3\}$, $\eta_{\{v_1,v_2\}}=\eta_{\{v_2,v_3\}}=1$,
and $\xi_{v_1}=\xi_{v_3}=1$ (i.e., a heat bath at each endpoint).  For~$\cS$ we take the same arbitrary positive reference heatbath $h=\hrf$
at both endpoints, and consider the alternative heat baths~$h'$ given by $h'_{v_1}=(1,0)$ and $h'_{v_3}=(0,1)$.  Here it can be checked that
$\spec(\cS+\cA)=\spec(\cS)=(4,2,2,2,2\pm \sqrt{2},1\pm \sqrt{3})$, 
as predicted
by Theorem~\ref{septhm}.  However, $\cS$ is reversible and hence diagonalizable, but 
$\cS+\cA$ has a $2 \times 2$ Jordan block within the three-dimensional space associated with eigenvalue~$2$.
\end{remark}

In light of Lemma~\ref{STA}, to prove spectral independence for the SEP as stated in Theorem~\ref{septhm}, 
it suffices to verify the hypotheses~\eqref{STcomm},~\eqref{ATcomm} of Lemma~\ref{STA} for the SEP.
To do this, it will be helpful to replace these conditions with local conditions that are easier to verify.

\begin{lemma} 
Consider the dynamics $\cS=\cSin+\cH$ with heat bath~$\cH$, as in~\eqref{Sin} and~\eqref{Hh},
and let $\cA=\cH'-\cH$ for some other heat bath~$\cH'$ as in~\eqref{A}.  Let $T$ be a $\{0,1\} \times \{0,1\}$ matrix such that 
\begin{align} [S,T\otimes I + I \otimes T]&=0 \label{STassump}; \\
\forall v \; [H(v),T]&=0; \label{HTassump} \\
\forall v \; [A(v),T]&=cA(v), \label{ATassump}
\end{align}
and define
\be \cT=\sum_{v \in V} T(v) \otimes I(V\backslash \{v\}), \label{T} \ee
where $T(v)$ is the operator applying $T$ at site $v$. 
Then~\eqref{STcomm} and~\eqref{ATcomm} hold.
\label{liftLemma} 
\end{lemma}

\begin{proof} 
To verify~\eqref{STcomm}, observe first that for $w \notin \{u,v\}$, $[S(\{u,v\}) \otimes I, I \otimes I \otimes T(w)]=0$. Consequently $[\cSin,\cT]=\sum_{\{u,v\}} \eta_{\{u,v\}} [S(\{u,v\}) , T(u)\otimes I + I \otimes T(v)] \otimes I(V\backslash \{u,v\})=0$ by~\eqref{STassump}.  Next, observe that for $w\neq v$, $[H(v)\otimes I, I\otimes T(w)]=0$. So
$[\cH,\cT] = 
\sum_{v \in V} \xi_{v} \cdot [H(v), T(v)] \otimes I(V\backslash \{v\}) =0$ by~\eqref{HTassump}. 

Putting these two observations together gives $[\cS,\cT]=[\cSin,\cT]+[\cH,\cT] = 0$, which is~\eqref{STcomm}.

To verify~\eqref{ATcomm}, note that for $u\neq v$, $[A(u)\otimes I,I\otimes T(v)]=0$. Consequently we have
$[\cA,\cT]=\sum_v \xi_v [A(v),T(v)] \otimes I(V\backslash \{v\}) = \sum_v \xi_v c A(v) \otimes I(V\backslash \{v\}) = c \cA$ by~\eqref{ATassump}.
\end{proof}

\begin{remark}
We note that the framework above for proving spectral equivalence is general and provides
sufficient conditions that may be satisfied by a broader class of processes.  Below we prove that
the conditions in Lemma~\ref{liftLemma} are satisfied by the SEP.  In a future version of this paper,
we will give some further examples of graph processes, of the general form in~\eqref{cSdef} with
internal dynamics~$\cSin$ and heat baths~$\cH$, that also satisfy these conditions, for a suitable
class of heat bath differences~$\cA$. 
\end{remark}

We are now in a position to prove Theorem~\ref{septhm}.

\begin{proofof}{Theorem~\ref{septhm}}
Theorem~\ref{septhm} claims spectrum equivalence of two arbitrary SEP dynamics
$\cS=\cSin+\cH$ and $\cS+\cA$, where $\cA=\cH'-\cH$.  Rather than prove this equivalence
directly, we instead show equivalence of the dynamics~$\cS+\cA$ with a fixed reference
dynamics~$\cS$.  Since any SEP dynamics with the same internal dynamics~$\cSin$ can be 
represented as~$\cS+\cA$, this establishes the theorem.

In the reference dynamics we will set the vertex heat baths to be $h_v=\hrf=(1/2,1/2)$ for all~$v$.

We begin by noting, via a short verification employing~\eqref{exchop}, that the
internal dynamics operator~$S$ in the SEP
satisfies the following balance condition for all $x_1,x_2,y_1,y_2\in Z$:
\be
\sum_{z} (S_{x_1,x_2}^{y_1,z}+S_{x_1,x_2}^{z,y_2}
-S_{z,x_2}^{y_1,y_2} - S_{x_1,z}^{y_1,y_2})  = 0
\label{T1cond1}
\ee

Now set $T=-\frac{1}{2}\one \cdot \one^\da$. 
In view of Lemmas~\ref{STA} and~\ref{liftLemma}, the proof will follow from showing~\eqref{STassump}--\eqref{ATassump}. 
First,~\eqref{STassump}:
\begin{align*} 
(S(T\otimes I + I \otimes T))_{x_1,x_2}^{y_1,y_2}
&=-\frac{1}{2} \sum_{z} (S_{x_1,x_2}^{z,y_2}+S_{x_1,x_2}^{y_1,z}); \\ 
 ((T\otimes I + I \otimes T)S)_{x_1,x_2}^{y_1,y_2}
 &=-\frac{1}{2}\sum_{z} (S_{z,x_2}^{y_1,y_2} + S_{x_1,z}^{y_1,y_2}), \\ 
 \noalign{and therefore, applying~\eqref{T1cond1},}
 [S,T\otimes I + I \otimes T]_{x_1,x_2}^{y_1,y_2}
 &=-\frac{1}{2} \sum_{z} (S_{x_1,x_2}^{y_1,z}+S_{x_1,x_2}^{z,y_2}
-S_{z,x_2}^{y_1,y_2} - S_{x_1,z}^{y_1,y_2}) =0.
\end{align*}

Now, for~\eqref{HTassump}, set 
\be H(v)=\one \cdot \hrf=\frac12 \one \cdot \one^\da \label{H'ref} \ee
for every $v$. Then $[H(v),T]=\frac{-1}{4}[\one \cdot \one^\da,\one \cdot \one^\da]=0$. 

Finally, for~\eqref{ATassump}, we note that since every heat bath $h_v$ is a 2-point probability distribution, 
every heat bath difference is of the form $A(v)= \one \cdot a_v$ where $a_v=(c,-c)$ for some $c\in\rset$. 
With our choice of~$T$, we therefore have
\begin{align*} [A(v),T]
&=-\frac{1}{2}((\one \cdot a_v \cdot \one \cdot \one^\da) - (\one \cdot \one^\da \one \cdot a_v)) \\
&= \frac{1}{2}((\one \cdot 0 \cdot \one^\da)-2(\one \cdot a_v)) \\
&= \one \cdot a_v \\ &= A(v),
\end{align*}
which establishes~\eqref{ATassump} with $c=1$.  This completes the proof.
\end{proofof}

\section*{Acknowledgments}
Thanks to Eldad Afik for helpful conversations.
\bibliographystyle{plainurl}
\bibliography{refs}

\end{document}